\documentclass[11pt]{article}
\usepackage{amsmath, amscd, amssymb, latexsym, epsfig, color, amsthm, tikz-cd}
\usepackage[all]{xy}
\usepackage{verbatim}
\numberwithin{equation}{section}

\newtheorem{theorem}{Theorem}[section]
\newtheorem{proposition}[theorem]{Proposition}
\newtheorem{question}[theorem]{Question}
\newtheorem{corollary}[theorem]{Corollary}
\newtheorem{conjecture}[theorem]{Conjecture}
\newtheorem{lemma}[theorem]{Lemma}

\theoremstyle{definition}

\newtheorem{remark}[theorem]{Remark}

\DeclareMathOperator{\skel}{Skel}
\DeclareMathOperator\lk{\mathrm{lk}}

\DeclareMathOperator\st{\mathrm{st}}

\DeclareMathOperator{\supp}{\mathrm{supp}}

\newcommand{\F}{{\mathbb F}}
\newcommand{\R}{{\mathbb R}}
\newcommand{\Q}{{\mathbb Q}}
\newcommand{\Z}{{\mathbb Z}}

\newcommand{\C}{{\mathcal C}}

\newcommand{\Stress}{\mathcal{S}}

\title{Simplicial spheres with $g_k=1$}
\author{
	Isabella Novik\thanks{Research of IN is partially\textsl{} supported by NSF grant  DMS-2246399.}\\
	\small Department of Mathematics\\[-0.8ex]
	\small University of Washington\\[-0.8ex]
	\small Seattle, WA 98195-4350, USA\\[-0.8ex]
	\small \texttt{novik@uw.edu}
	\and 
	Hailun Zheng\thanks{Research of HZ is partially\textsl{} supported by NSF grant DMS-2535689.} \\
	\small Department of Mathematics\\[-0.8ex]
	\small University of Hawai`i at M\={a}noa\\[-0.8ex]
	\small 2565 McCarthy Mall, Honolulu, HI 96822, USA \\[-0.8ex]
	\small \texttt{hailunz@hawaii.edu}
}			
\begin{document}
	\maketitle
	\begin{abstract}
		For $d\geq 4$, Kalai (1987) characterized all simplicial $(d-1)$-spheres with $g_2=0$, and for $k\geq 2$ and $d\geq 2k$, Murai and Nevo (2013) characterized all simplicial $(d-1)$-spheres with $g_k=0$. In addition, for $d\geq 4$, Nevo and Novinsky (2011) characterized all simplicial $(d-1)$-spheres with $g_2=1$. Motivated by these results, we characterize, for any $k\geq 2$ and $d\geq 2k+1$, all simplicial $(d-1)$-spheres with no missing faces of dimension larger than $d-k$ that satisfy $g_k=1$. When $d=2k$, we obtain a characterization of simplicial $(d-1)$-spheres with $g_k=1$ and no missing faces of dimension greater than $k$, under the additional assumption that there exists at least one missing face of dimension $k$. Finally, for $k=3$, we are able to remove this assumption and characterize all simplicial $5$-spheres with no missing faces of dimension larger than $3$ that satisfy $g_3=1$.
	\end{abstract}
	\section{Introduction}
	
	What is the smallest number of edges that a $(d-1)$-dimensional simplicial sphere with $f_0$ vertices can have? When $d\geq 3$, the answer is provided by the Lower Bound Theorem (LBT) \cite{Barnette-LBT-pseudomanifolds, Barnette73, Kalai87}, which asserts that $f_1\geq df_0-\binom{d+1}{2}$. The quantity $f_1-df_0+\binom{d+1}{2}$ is denoted by $g_2$. The inequality $g_2\geq 0$  holds not only for simplicial spheres but for all normal pseudomanifolds; see \cite{Kalai87, Fogelsanger88,Tay}. Furthermore, the case $g_2 = 0$ is completely characterized: when $d\geq 4$, $g_2=0$ holds if and only if the normal $(d-1)$-pseudomanifold in question is the boundary complex of a stacked $d$-polytope.
	
	Motivated by these results, Nevo and Novinsky \cite{NevoNovinsky} gave, for $d\geq 4$, a complete characterization of simplicial $(d-1)$-spheres with no missing faces of dimension larger than $d-2$ that satisfy $g_2=1$.
	Since then, several additional results characterizing simplicial manifolds (or even normal pseudomanifolds) with small values of $g_2$ have appeared in the literature; see, for example, \cite{Z-g_2leq2, BasakSwartz}.
	
	In complete analogy with the  LBT, for any $1\leq k\leq d/2$, McMullen and Walkup \cite{McMullenWalkup71} conjectured a lower bound on the number of $(k-1)$-faces, $f_{k-1}$, in terms of the face numbers $f_{k-2},\dots,f_0, f_{-1}$. Using the notion of the $g$-numbers---certain alternating linear combinations of the $f$-numbers---their conjecture asserts that for a simplicial $(d-1)$-sphere $\Delta$ and any $k\leq d/2$, we have $g_k(\Delta)\geq 0$. Moreover, equality $g_k=0$ holds if and only if $\Delta$ is $(k-1)$-stacked. This conjecture is known as the Generalized Lower Bound Conjecture (GLBC). The inequality part of the GLBC for the boundary complexes of simplicial polytopes was proved by Stanley~\cite{Stanley80} as part of his proof of the celebrated $g$-theorem. Much more recently, these inequalities and, more generally, the $g$-theorem were established for all simplicial spheres, and even for a larger class of $\mathbb{Z}/2\mathbb{Z}$-homology spheres; see~\cite{Adiprasito-g-conjecture,PapadakisPetrotou,AdiprasitoPapadakisPetrotou,KaruXiao}. The equality part of the GLBC was proved by Murai and Nevo~\cite{MuraiNevo2013}.
	
	The goal of this paper is to understand, for a given $k$, spheres with no large missing faces that satisfy $g_k = 1$. Throughout the paper, we work with $\mathbb{Z}/2\mathbb{Z}$-homology spheres and often refer to them simply as spheres. In particular, this class contains all simplicial spheres. We denote by $S(j,d-1)$ the collection of $\mathbb{Z}/2\mathbb{Z}$-homology $(d-1)$-spheres all of whose missing faces have dimension at most $j$. For instance,  $S(1,d-1)$ coincides with the class of flag $(d-1)$-spheres. 
	
	Denote by $\partial\sigma^i$ the boundary complex of an $i$-simplex. Our main results can be summarized as follows:
	\begin{enumerate}
	\item Let $d\geq 2k$, and let $\Delta\in S(d-k,d-1)$. Assume further that if $d=2k$, then $\Delta$ has at least one missing $k$-face; that is, 
	$\Delta\notin S(d-k-1, d-1)=S(k-1, 2k-1)$. Then $g_k(\Delta)=1$ if and only if $\Delta$ is either the join of $\partial \sigma^{d-k}$ and a $(k-1)$-sphere, or the join of $\partial \sigma^j$ and $\partial \sigma^{d-j}$ for some $k < j\leq \lfloor d/2\rfloor$; see Theorem~\ref{main-thm:g=1}.
	\item In the case $d=2k=6$, we give a complete characterization of spheres in $S(3,5)$ with $g_3=1$: $\Delta\in S(3,5)$  has $g_3=1$ if and only if $\Delta$ is either the join of $\partial \sigma^3$ and a $2$-sphere, or the join of three copies of $\partial \sigma^2$; see Theorem~\ref{main-thm: S(2,5)}. In particular, all such spheres are boundaries of simplicial polytopes. 
	\end{enumerate}
	
	For $k=2$, our first result recovers the theorem of Nevo and Novinsky \cite{NevoNovinsky}.\footnote{The condition that for $d=4$, $\Delta\notin S(1,3)$ is automatically satisfied, as all spheres in $S(1,3)$ have $g_2\geq 2$.}  Our proofs rely on the theory of (higher) stress spaces developed by Lee \cite{Lee96} and by Tay, White, and Whiteley \cite{Tay-et-al-I,Tay-et-al}. More specifically, we establish a new version of the cone lemma that allows us to describe the supports of certain stresses (see Section 3), which, in turn, enables us to glean some information about the structure of the complex in question. Another tool we use is McMullen's integral formula \cite[Proposition 2.3]{Swartz05}.
	
The structure of the paper is as follows. In Section~2, we review simplicial complexes and introduce the main object of the paper, $S(j,d-1)$, along with the requisite background on Stanley--Reisner rings and stress spaces; we also derive several corollaries concerning the $g$-numbers. In Section~3, we establish a version of the cone lemma and discuss one of its applications---Lemma~\ref{lm1}.
Sections~4 and 5 are devoted to proving the two main theorems of the paper. We conclude in Section~6 with a discussion of related open problems.
	
	\section{Preliminaries}
	\subsection{Simplicial complexes and face numbers}
	An (abstract) {\em simplicial complex} $\Delta$ with vertex set $V=V(\Delta)$ is a collection of subsets of $V$ that is closed under inclusion and contains all singletons; that is, $\{v\}\in\Delta$ for all $v\in V$. An example of a simplicial complex on $V$ is the collection of all subsets of $V$. When $|V|=d+1$, this complex is a {\em $d$-simplex}, and we usually denote it by $\sigma^d$ or by $\overline{V}$ when the vertex set of this simplex is important.
	
	The elements of a simplicial complex $\Delta$ are called {\em faces} of $\Delta$. A face $\tau$ of $\Delta$ has {\em dimension} $i$ if $|\tau|=i+1$; in this case we say that $\tau$ is an {\em $i$-face}. We usually refer to $0$-faces as {\em vertices}, $1$-faces as {\em edges}, and the maximal faces under inclusion as {\em facets}. For brevity, we denote a vertex by $v$, an edge by $uv$, a $2$-face by $uvw$, instead of $\{v\}$, $\{u, v\}$, and $\{u,v,w\}$, respectively. The {\em dimension of $\Delta$} is $\max\{\dim \tau: \tau\in \Delta\}$. We say that $\Delta$ is {\em pure} if all facets of $\Delta$ have the same dimension.

	A set $\tau\subseteq V$ is a {\em missing face} of $\Delta$ if $\tau$ is not a face of $\Delta$, but every proper subset of $\tau$ is a face of $\Delta$. In analogy with faces, a {\em missing $i$-face} is a missing face of size $i+1$. The collection of the missing faces of $\Delta$, together with its vertex set, uniquely determines $\Delta$. For instance, the only simplicial complex on a vertex set $V$ with  no missing faces is the simplex $\overline{V}$.
	
	Let $\tau$ be a face of $\Delta$. The {\em star} and {\em link} of $\tau$ are defined as
	$$\st(\tau)=\st(\tau,\Delta)=\{\sigma \in \Delta \ : \  \sigma\cup \tau\in\Delta\}, \qquad\lk(\tau)=\lk(\tau,\Delta)= \{\sigma\in \st(\tau) \ : \ \sigma\cap \tau=\emptyset\}.$$ When $\tau=v$ is a vertex, we also define $\Delta\backslash v=\{\sigma\in\Delta : v\notin\sigma\}$; this subcomplex of $\Delta$ is called the {\em antistar} of $v$.
	
	A subcomplex of $\Delta$ is called {\em induced} if it is of the form $\Delta[W]=\{\tau\in \Delta: \tau\subseteq W\}$ for some $W\subseteq V(\Delta)$. The subcomplex of $\Delta$ consisting of all faces of $\Delta$ of dimension $\leq k$ is called the {\em $k$-skeleton} of $\Delta$ and is denoted $\skel_k(\Delta)$; the $1$-skeleton of $\Delta$ is also known as the {\em graph} of $\Delta$. Finally, if $\Delta$ and $\Gamma$ are two simplicial complexes on disjoint vertex sets, then their {\em join} is 
	$$\Delta*\Gamma=\{\sigma\cup \tau: \sigma\in \Delta, \tau \in \Gamma\}.$$ If $\Gamma$ is a $0$-simplex, that is, $\Gamma=\{v,\emptyset\}$, we write $\Delta*\Gamma=\Delta*v$  and call this complex the {\em cone} over $\Delta$ with apex $v$.
	 
	 Let $\F$ be a field. A pure $(d-1)$-dimensional simplicial complex is an {\em $\F$-homology sphere} if, for every face $\sigma$ (including the empty face), the link $\lk(\sigma)$ has the homology of a $(d-1-|\sigma|)$-dimensional sphere (over $\F$). 
	 Denote by $\|\Delta\|$ the geometric realization of $\Delta$. We say that $\Delta$ is a {\em simplicial $(d-1)$-sphere} if $\|\Delta\|$ is
	 homeomorphic to a $(d-1)$-dimensional sphere. It is known that a $\Z/2\Z$-homology sphere is always an $\R$-homology sphere; see \cite[Lemma 2.1]{KaruXiao}. Moreover, a simplicial sphere is an $\F$-homology sphere for any field $\F$.
	 
	 In this paper, we work with the class of $\Z/2\Z$-homology spheres, which we refer to simply as spheres.\footnote{The main reason for working with this class is that, unlike the class of simplicial spheres, it is closed under taking links.} Let $\Delta$ be a $(d-1)$-sphere. Denote by $f_i=f_i(\Delta)$ the number of $i$-faces of $\Delta$ and by $m_i=m_i(\Delta)$ the number of missing $i$-faces of $\Delta$. In particular, $f_i=0$ if $i>d-1$, and $m_j(\Delta)>0$ for at least one $1\leq j\leq d$. Let $$f(\Delta)=(f_{-1}, f_0, \dots, f_{d-1}) \quad\mbox{and}\quad m(\Delta)=(m_1, m_2, \dots, m_d)$$
	 be the {\em $f$-vector} and the {\em $m$-vector} of $\Delta$, respectively. The {\em $h$-vector} of $\Delta$, $h(\Delta)=(h_0, h_1,\ldots,h_d)$, is obtained from the $f$-vector by the following invertible linear transformation:
	 $$h_j= h_j(\Delta)=\sum_{i=0}^{j} (-1)^{j-i}\binom{d-i}{d-j}f_{i-1}(\Delta) \quad \text{for }\, 0\leq j\leq d.$$ 
	The {\em $g$-vector} of $\Delta$, $g(\Delta)=(g_0,g_1,\ldots, g_{\lfloor d/2\rfloor})$, is defined by letting $g_0=1$ and $g_j=h_j-h_{j-1}$ for $1\leq j\leq \lfloor d/2\rfloor$. When $d$ is odd, we sometimes also consider $g_{\lceil d/2\rceil}$.
	
	The Dehn--Sommerville relations \cite{Klee64} assert that the $h$-vector of a $(d-1)$-sphere is symmetric: $h_i=h_{d-i}$ for all $0\leq i\leq d$. In particular, it follows that every $2k$-sphere satisfies $g_{k+1}=0$.
	 
	 Following the notation in \cite{Nevo2009}, let $S(i,d-1)$ denote the class of $\Z/2\Z$-homology $(d-1)$-spheres all of whose missing faces have dimension at most $i$. Recall that a simplicial complex is \emph{flag} if it has no missing faces of dimension greater than~$1$, or equivalently, if it is the clique complex of its $1$-skeleton. Thus, $S(1,d-1)$ is precisely the class of flag $(d-1)$-spheres.
	
	The following element of $S(i, d-1)$ deserves special attention. Write $d=qi+r$, where $q$ and $r$ are (uniquely defined) integers satisfying  $1\leq r\leq i$, and let $K(i, d-1)=(\partial \sigma^i)^{*q} * \partial \sigma^r$, where $(\partial \sigma^i)^{*q}$ denotes the join of $q$ copies of $\partial \sigma^i$. It is known that the sphere $K(i, d-1)$ simultaneously minimizes all the $f$- and $h$-numbers among all spheres in $S(i,d-1)$; see \cite{GoffKleeN, Nevo2009}. Furthermore, $K(1, d-1)$, the boundary complex of the $d$-cross-polytope, simultaneously minimizes all the $g$-numbers among all flag PL $(d-1)$-spheres; see \cite{NZ-Aff-Reconstr}.
	
	To close this section, we mention that, in analogy with simplicial spheres and $\F$-homology spheres, one can define {\em simplicial balls} and {\em $\F$-homology simplicial balls}. The link of any face $\tau$ of an $\F$-homology ball $B$ is either an $\F$-homology sphere or an $\F$-homology ball. We call $\tau$ an {\em interior face} in the former case and a {\em boundary face} in the latter. A {\em minimal interior face}  is an interior face that contains no other interior faces. The collection of all boundary faces of $B$ forms an $\F$-homology sphere of dimension one less than that of $B$. This sphere is called the {\em boundary complex} of $B$ and is denoted $\partial B$.

	In this paper, we use only $\R$-homology balls, which we often refer to simply as {\em balls}. An $\R$-homology $(d-1)$-sphere $\Delta$ is called {\em $(i-1)$-stacked} if there exists an $\R$-homology $d$-ball $B$ such that $\partial B=\Delta$ and $B$ has no interior faces of dimension $\leq d-i$. Such a ball $B$ is called an {\em $(i-1)$-stacked triangulation} of $\Delta$.
	
The following simple result will be useful.

\begin{lemma} \label{lm:ball-with-one-interior-face}
If $B$ is a ball with a unique minimal interior face $\tau$, then $B=\overline{\tau}*\lk(\tau,B)$.
\end{lemma}

\begin{proof}
Every facet of $B$ is an interior face and hence contains a minimal interior face, namely $\tau$. Since $B$ is pure, every face of $B$ is contained in a facet, and therefore lies in the star of $\tau$. Hence
$B=\st(\tau, B)=\overline{\tau}*\lk(\tau,B).$
\end{proof} 
	
	\subsection{The Stanley--Reisner ring}
	Let $\Delta$ be a $(d-1)$-dimensional simplicial complex with vertex set $V=V(\Delta)$. 
	Let $\F$ be a field of characteristic zero, and let $X=\{x_v : v\in V\}$ be a set of variables, one for each vertex. Denote by $\F[X]=\F[x_v : v\in V] $ the polynomial ring over $\F$ in these variables. The {\em Stanley--Reisner ideal} of $\Delta$ is the ideal of $\F[X]$ generated by the monomials corresponding to missing faces of $\Delta$:
	\[I_\Delta=(x_{j_1}x_{j_2}\cdots x_{j_k} : \{j_1,\ldots,j_k\} \mbox{ is a missing face of }\Delta).\]
	
	The {\em Stanley--Reisner ring} (or {\em face ring}) of $\Delta$ is the quotient $\F[\Delta]:=\F[X]/I_\Delta$. This is a graded ring. Its Hilbert series is given by $\big(\sum_{i=0}^d h_i(\Delta) t^i\big)/(1-t)^d$; see \cite[Theorem II.1.4]{Stanley96}. A sequence $\Theta$ of $d$ linear forms $\theta_1,\ldots,\theta_d$ in $\F[\Delta]$ is called a {\em linear system of parameters} (l.s.o.p.) for $\F[\Delta]$ if the quotient ring $\F[\Delta]/(\theta_1,\ldots,\theta_d)$ is a finite-dimensional $\F$-vector space; this quotient ring is then called an {\em artinian reduction} of $\F[\Delta]$.

	In what follows, assume that $\Delta$ is a $\Z/2\Z$-homology $(d-1)$-sphere. The Stanley--Reisner ring of $\Delta$ is {\em Cohen--Macaulay} \cite{Reisner}; hence any l.s.o.p.~$\Theta$ for $\F[\Delta]$ is a {\em regular sequence}, that is,  for each $1\leq i\leq d$, $\theta_i$ is a non-zero-divisor on $\F[\Delta]/(\theta_1,\ldots,\theta_{i-1})$. Consequently, $\dim_\F \big(\F[\Delta]/(\Theta)\big)_i=h_i(\Delta)$ for all $0\leq i\leq d$, where $R_i$ denotes the $i$-th graded component of a graded ring $R$. The fact that any artinian reduction of the Stanley--Reisner ring of a $(d-1)$-sphere $\Delta$ is a Poincar\'e duality algebra then provides an alternative proof of the Dehn--Sommerville relations: $h_i(\Delta)=h_{d-i}(\Delta)$ for all $0\leq i\leq d$. 
	
	In this paper, we take $\{a_{v,j}\in\R: v\in V, j\in[d]\}$ to be a set that is {\bf algebraically independent} over $\Q$, and we let $\F=\Q(a_{v, j} : v\in V, j\in[d])$. The sequence $\Theta=(\theta_1, \dots, \theta_d)$, where $\theta_j:=\sum_{v\in V} a_{v,j}x_v$ for $1\leq j\leq d$, then forms an l.s.o.p.~for $\F[\Delta]$; see \cite[Theorem III.2.4]{Stanley96}. The ring $\F[\Delta]/(\Theta)$ is called the {\em generic artinian reduction} of $\F[\Delta]$.

	Let $c=\sum_{v\in V} x_v$. The following result is the celebrated (algebraic version of the) $g$-theorem; see \cite{McMullen96,Stanley80} for the case of simplicial polytopes and \cite{Adiprasito-g-conjecture,PapadakisPetrotou, AdiprasitoPapadakisPetrotou,KaruXiao}, or, more specifically, \cite[Theorem 1.3]{KaruXiao}, for the case of spheres.

	\begin{theorem} \label{alg-g-thm}
	Let $\Delta$ be a $(d-1)$-sphere, let  $\F[\Delta]/(\Theta)$ be the generic artinian reduction of $\F[\Delta]$, and let $c=\sum_{v\in V} x_v$.
		Then, for every $0\leq k\leq \lfloor d/2\rfloor$, the map $\cdot c^{d-2k}: \big(\F[\Delta]/(\Theta)\big)_k \to \big(\F[\Delta]/(\Theta)\big)_{d-k}$ is an isomorphism. In particular, the map
			$$\cdot c: \big(\F[\Delta]/(\Theta)\big)_k \to \big(\F[\Delta]/(\Theta)\big)_{k+1}$$ is injective for all $0\leq k\leq \lceil d/2\rceil -1$ and surjective for all $\lfloor d/2\rfloor \leq k\leq d-1$. Consequently, $\dim_\F  \big(\F[\Delta]/(\Theta, c)\big)_k=g_k(\Delta)$ for all $k\leq \lceil d/2\rceil$.
	\end{theorem}
	
	\noindent In view of this theorem, $c=\sum_{v\in V} x_v$ is called the {\em canonical Lefschetz element}. The inequality part of the Generalized Lower Bound Theorem (GLBT, for short) then follows immediately:
	
	\begin{theorem}  \label{GLBT-ineq}
	If $\Delta$ is a $(d-1)$-sphere, then $g_k(\Delta)\geq 0$ for all $k\leq \lfloor d/2\rfloor$.
	\end{theorem}
	
	Since $\F[\Delta]/(\Theta,c)$ is a standard graded algebra, Macaulay's theorem \cite[Theorem II.2.2]{Stanley96} implies that the $g$-numbers form an $M$-sequence, that is, they also satisfy certain nonlinear inequalities. Specifically, by \cite[Theorem~II.2.2]{Stanley96}, if $k<\lfloor d/2\rfloor$ and $g_k = {a \choose k}$ for a nonnegative integer $a$, then $g_{k+1} \le {a+1 \choose k+1}$. This leads to the following well-known statement:
	\begin{proposition} \label{prop:1and0imply1and0}
		Let $\Delta$ be a $(d-1)$-sphere. If $g_k(\Delta)= 1$ for some $k<\lfloor d/2\rfloor$, then for all $k<j\le \lfloor d/2\rfloor$, we have $g_j(\Delta) \le 1$. Furthermore, if $g_k(\Delta)= 0$ for some $k<\lfloor d/2\rfloor$, then $g_j(\Delta)= 0$ for all $k<j\le \lfloor d/2\rfloor$.
	\end{proposition}
	
	\subsection{More results on the $g$-numbers}
	Now we collect a few additional results on the $g$-numbers of spheres. We start with the following elementary but extremely useful lemma, which applies to all pure simplicial complexes. It was established for simplicial polytopes by McMullen \cite[p.~183]{McMullen70} and for all pure simplicial complexes by Swartz \cite[Proposition 2.3]{Swartz05}, and is known in the literature as {\em McMullen's integral formula}.
	\begin{lemma}\label{lm: Swartz}
		If $\Delta$ is a pure $(d-1)$-dimensional simplicial complex, then for $0\leq k\leq \lfloor \frac{d-1}{2}\rfloor$,
		\[\sum_{v\in V(\Delta)} g_k(\lk(v))=(k+1)g_{k+1}(\Delta)+(d+1-k)g_k(\Delta).\]
	\end{lemma}
	
The following result addresses the equality part of the GLBT. Given a simplicial complex $\Delta$, define $$\Delta(j)=\{\tau\subseteq V(\Delta): \skel_{j}(\overline{\tau})\subseteq \Delta\}.$$
	\begin{theorem}\label{thm: GLBT}
		Let $1\leq k\leq d/2$ and let $\Delta$ be a $(d-1)$-sphere. Then $g_k(\Delta)=0$ if and only if $\Delta$ is $(k-1)$-stacked. Furthermore, if $g_k(\Delta)=0$, then $\Delta(d-k)=\Delta(k-1)$, and this complex is an $\R$-homology $d$-ball that provides a $(k-1)$-stacked triangulation of $\Delta$. 
	\end{theorem}
	
	\noindent The implication $g_k=0 \Rightarrow \Delta(d-k)=\Delta(k-1)$ is due to Nagel \cite[Corollary~4.8]{Nagel}. The implication $g_k(\Delta)=0\Rightarrow \Delta\text{ is $(k-1)$-stacked}$ is due to Murai and Nevo \cite[Theorem~1.3]{MuraiNevo2013}. Finally, the assertion that $\Delta$ being $(k-1)$-stacked implies $g_k=0$ is due to McMullen and Walkup~\cite{McMullenWalkup71}.
	
	We will also use the following characterization of $k$-stackedness in terms of the $m$-numbers, proved in \cite[Corollary 1.4]{MNZ}:
	\begin{proposition}\label{thm2: Murai-Novik-Zheng}
		Let $d\geq 4$ and let $\Delta$ be a $(d-1)$-sphere. Then for $1\leq k\leq \lfloor\frac{d}{2}\rfloor  -1$, $\Delta$ is $k$-stacked if and only if $g_k(\Delta)=m_{d-k}(\Delta)$. 
	\end{proposition}
	
	Proposition~\ref{thm2: Murai-Novik-Zheng} was motivated by the following result of Nagel \cite[Corollary 4.6(a)]{Nagel}
	
	\begin{proposition}\label{prop: enhanced upper bound g-theorem}
		Let $d\geq 4$ and let $\Delta$ be a $(d-1)$-sphere. Then for all $1\leq k\leq \lceil\frac{d}{2}\rceil -1$, $g_k(\Delta)\geq m_{d-k}(\Delta)$.
\end{proposition}	
	
 \begin{remark}\label{rm: g_k>0}
Let $k\leq d/2$ and assume that $\Delta$ is a $(d-1)$-sphere with $g_k=0$. Then, by the part of Theorem~\ref{thm: GLBT} asserting that $\Delta(k-1)=\Delta(d-k)$, $\Delta$ has no missing $j$-faces for any $k\leq j\leq d-k$. 
On the other hand, since $\Delta(d-k)$ is $d$-dimensional  (see Theorem~\ref{thm: GLBT}) and therefore strictly contains $\Delta$, it follows that $\Delta$ must have a missing face of dimension $\geq d-k+1$. In particular, any sphere in $S(d-k,d-1)$ satisfies $g_k\geq 1$.
 \end{remark}

In this paper, we are interested in $(d-1)$-spheres with $g_k=1$. By combining Propositions~\ref{thm2: Murai-Novik-Zheng} and \ref{prop: enhanced upper bound g-theorem} with Theorem~\ref{thm: GLBT}, we obtain the following result:

	\begin{corollary} \label{cor:Nagel}
	Let $k< \lfloor d/2\rfloor$ and let $\Delta$ be a $(d-1)$-sphere with $g_k(\Delta)=1$. Then $m_{d-k}(\Delta)\leq 1$. Furthermore, when $k\leq \lfloor d/2\rfloor-1$, we have $m_{d-k}(\Delta)=1$ if and only if $g_{k+1}(\Delta)=0$.
	\end{corollary}

	\subsection{The stress spaces}
	In this subsection, we review the basics of linear and affine stresses. For further details, we refer the reader to \cite{Lee94,Lee96} and \cite{Tay-et-al-I, Tay-et-al}, and to \cite[Section 2.3]{MNZ} for a more algebraic exposition.
	
	Let $\F\subseteq \R$ be a field, and let $\Delta$ be a $(d-1)$-dimensional simplicial complex with vertex set $V=V(\Delta)$. A map $p: V(\Delta) \rightarrow \F^d\subseteq \R^d$ is called a \textit{$d$-embedding} of $\Delta$. 
	
	Continuing with the notation of Section 2.2, let $X =\{x_v : v\in V\}$ and let $\F[X]$ be the corresponding ring of polynomials. Each variable $x_v$ acts on $\F[X]$ by $\frac{\partial}{\partial{x_v}}$; for brevity, we denote this operator by $\partial_{x_v}$. More generally, if $\mu=x_{i_1}\cdots x_{i_s}\in \F[X]$ is a monomial, define $\partial_\mu : \F[X] \to \F[X]$ by $\rho \mapsto \partial_{x_{i_1}}\cdots\partial_{x_{i_s}}\rho$. If $\ell =\sum_{v\in V} \ell_v x_v$ is a linear form in $\F[X]$, define
	$$\partial_{\ell} : \F[X]\to\F[X] \quad \mbox{by} \quad \rho \mapsto \sum_{v\in V}\ell_v\cdot\partial_{x_v}\rho=\sum_{v\in V}\ell_v\frac{\partial \rho}{\partial{x_v}}.$$
	For a monomial $\mu\in \F[X]$, the {\em support} of $\mu$ is $\supp(\mu)=\{v\in V: \;x_v|\mu\}$.
	
	A $d$-embedding $p$ of $\Delta$ gives rise to the sequence  $\Theta(p)=(\theta_1, \dots, \theta_d)$ of $d$ linear forms, where $\theta_j=\sum_{v\in V} p(v)_j\, x_v$. Now, let $\F=\Q(a_{v, j}: v\in V, j\in [d])$, where $\{a_{v, j}\in\R: v\in V, j\in [d]\}$ is algebraically independent over $\Q$. In this case, the embedding $p: V \to \F^d$ defined by $v\mapsto p(v)=(a_{v,1},\dots, a_{v,d})$, is called a {\em generic} embedding. 
	
	A homogeneous polynomial $\lambda=\lambda(X)=\sum_\mu \lambda_\mu \mu\in\F[X]$ of degree $k$ is called a {\em linear  $k$-stress} on $(\Delta, p)$ if:
	\begin{itemize}
		\item Every nonzero term $\lambda_\mu \mu$ of $\lambda$ is supported on a face of $\Delta$, that is, $\supp(\mu)\in\Delta$, and
		\item $\partial_{\theta_i}\lambda=0$ for all $i=1,\ldots, d$.
	\end{itemize}
	Recall that $c=\sum_{v\in V}x_v$. A linear $k$-stress $\lambda$ on $(\Delta, p)$  that also satisfies $\partial_{c}\lambda=\sum_{v\in V}\partial_{x_v}\lambda=0$ is called an {\em affine $k$-stress}. 
	
	It follows immediately from the definitions that the sets of linear $k$-stresses and affine $k$-stresses on $\Delta$ form vector spaces over $\F$, which we denote by $\Stress^\ell_k(\Delta,p)$ and $\Stress^a_k(\Delta, p)$, respectively. When we wish to emphasize the underlying field, we write $\Stress^\ell_k(\Delta,p;\F)$ and $\Stress^a_k(\Delta, p;\F)$. Furthermore, for all $k$, the space $\Stress^a_k(\Delta,p)$ is the kernel of $\partial_{c} : \Stress^\ell_k(\Delta,p) \to  \Stress^\ell_{k-1}(\Delta,p)$.
	
	The spaces $\Stress^\ell_k(\Delta, p)$ and $\Stress^a_k(\Delta, p)$ coincide with  certain graded components of the Macaulay inverse system of  $I_\Delta+(\Theta(p))$ and $I_\Delta+(\Theta(p),c)$, respectively; see \cite[Section 2.3]{MNZ} for details. Consequently, the dimensions of $\Stress^\ell_k(\Delta,p)$ and  $\Stress^a_k(\Delta, p)$ agree with the dimensions of the $k$-th graded components of $\F[\Delta]/(\Theta(p))$ and $\F[\Delta]/(\Theta(p),c)$, respectively (see \cite{Lee96} or \cite{MNZ}). In particular, the following (weaker) restatement of the $g$-theorem (Theorem~\ref{alg-g-thm}) in the language of stress spaces holds:

	\begin{theorem}\label{thm: Lefschetz}
	Let $(\Delta, p)$ be a $(d-1)$-sphere with a generic embedding $p$, and let $0\leq k\leq \lceil d/2\rceil$. 
	Then $\dim \Stress^\ell_k(\Delta,p)=\dim \Stress^\ell_{d-k}(\Delta,p)=h_k(\Delta)$ and $\dim \Stress^a_k(\Delta,p)=g_k(\Delta)$. Moreover, if $k>0$, the linear map $\partial_{c}: \Stress^\ell_k(\Delta, p) \to \Stress^\ell_{k-1}(\Delta, p)$ is surjective. In particular, if $d$ is odd and $k=(d+1)/2$, then this map is an isomorphism.
	\end{theorem}
 
 	Let $\lambda= \sum_\mu \lambda_\mu \mu$ be either a linear $k$-stress or an affine $k$-stress on $(\Delta,p)$, and let $\sigma$ be a $(k-1)$-face of $\Delta$. We write $x_\sigma:=\prod_{v\in \sigma} x_v$ and $\lambda_\sigma:= \lambda_{x_\sigma}$. 
 	By \cite[Theorems 9, 11]{Lee96}, a $k$-stress $\lambda$ is uniquely determined by its squarefree part $\sum_{\sigma\in\Delta, |\sigma|=k}\lambda_\sigma x_\sigma$.
 	
	We write $\supp(\lambda)$ for the subcomplex of $\Delta$ generated by all $(k-1)$-faces $\sigma$ with $\lambda_\sigma \neq 0$. We also say that $\lambda$ {\em lives on}  a subcomplex $\Gamma\subseteq \Delta$ if $\supp(\lambda)\subseteq \Gamma$. For instance, $\partial_{x_v}\lambda$ is a $(k-1)$-stress that lives on $\st(v)$. Finally, we say that a face $\tau$ {\em participates in} $\lambda$, or that $\tau$ is {\em in the support of} $\lambda$, if $\tau\in \supp(\lambda)$.

	\section{The cone lemmas}
	Assume $\Lambda$ is a simplicial complex of dimension $d-2$ and let $\Gamma$ be the cone over $\Lambda$. It is known that, for appropriately chosen embeddings of $\Lambda$ in $\R^{d-1}$ and of $\Gamma$ in $\R^d$, the stress spaces of $\Lambda$ and $\Gamma$ are isomorphic; see, for instance, \cite[Theorem 10]{Lee96}, \cite[Lemma 3.2]{NZ-reconstruction}. Statements of this form are referred to as {\em cone lemmas} in the literature. The importance of the following version of the cone lemma is that it allows us to describe how the support of a stress on $\Lambda$ relates to the support of its image on $\Gamma$.
	
	\begin{lemma}\label{cone lemma-linear}
		Let $\Lambda$ be a $(d-2)$-dimensional simplicial complex with $V(\Lambda)=[n]$, and let $\Gamma=0*\Lambda$ be the cone over $\Lambda$. Let $p': [n]\to \R^{d-1}$ be an embedding of $\Lambda$ such that the set of coordinates $\{p'(s)_t : s\in [n], t\in [d-1]\}$ is algebraically independent over $\Q$, and define $\mathbb{F}'=\Q(p'(s)_t: s\in[n], \, t\in [d-1])$. Let $a_1, \dots, a_n\in \R$ be algebraically independent over $\mathbb{F}'$, and let $\F\subseteq \R$ be any field that contains $\mathbb{F}'(a_1, \dots, a_n)$. Define an embedding $p: \{0, 1, \dots, n\}\to\F^d$ of $\Gamma$ by $$p(0)=(0, \dots, 0, -1), \quad p(s)=((1+a_s)p'(s), a_s), \; \text{for}\; s\in [n].$$ Then, for all $1\leq i\leq d-1$, there exists a linear map of $\mathbb{F}'$-vector spaces $\psi_i: \Stress^\ell_i(\Lambda, p'; \mathbb{F}') \to \Stress^\ell_i(\Gamma, p; \mathbb{F})$ such that: 1) for any $\omega'\in \Stress^\ell_i(\Lambda, p'; \mathbb{F}')$, $\supp(\psi_i(\omega'))=\skel_{i-1}\big(0*\supp(\omega')\big)$, and 2) if $c'=\sum_{j=1}^n x_j$ and $c=\sum_{j=0}^n x_j$, then $\psi_{i-1}(\partial_{c'}\omega')=\partial_c(\psi_i(\omega'))$.
	\end{lemma}
	
	\noindent In other words, statement 2) implies that the following diagram commutes:
	\[ \begin{tikzcd}
		\Stress^\ell_i(\Lambda, p'; \F') \arrow{r}{\psi_i} \arrow[swap]{d}{\partial_{c'}} & \Stress^\ell_i(\Gamma, p; \F)  \arrow{d}{\partial_c} \\
		\Stress^\ell_{i-1}(\Lambda, p'; \F')\arrow{r}{\psi_{i-1}}& \Stress^\ell_{i-1}(\Gamma, p; \F)
	\end{tikzcd}.
	\]
	We refer to the map $\psi_i$ as the {\em lifting} map.
	
	\begin{proof} 
	The construction of $\psi_i$ is similar to that of \cite[Lemma 3.2]{NZ-reconstruction}, and so we omit some details.
		 In our setting, the l.s.o.p.~for $\F[\Gamma]$ corresponding to $p$ is given by $\theta_t=\sum_{s=1}^n (1+a_s)p'(s)_t x_s$, for $1\leq t\leq d-1$, and $\theta_d=-x_0+\sum_{s=1}^n a_s x_s$. Given $\omega'\in \Stress^\ell_i(\Lambda, p'; \mathbb{F}')$, define  polynomials in $\F[x_1,\dots,x_n]$ recursively by
		\begin{eqnarray} \nonumber
		\omega_0(x_1, \dots, x_n)&=&\omega'\left(\frac{x_1}{1+a_1}, \dots, \frac{x_n}{1+a_n}\right), \quad \mbox{and}\\
		 \label{omega_j}
			\omega_{j+1}&=&\frac{1}{j+1}\sum_{s=1}^n a_s \, \partial_{x_s}\omega_j, \;\; \forall \; 0\leq j\leq i-1.
		\end{eqnarray}
		In particular, $\omega_1=\sum_{s=1}^n a_s \, \partial_{x_s}\omega_0.$ We then define $$\psi_i(\omega'):=\sum_{j=0}^i x_0^j\cdot \omega_j(x_1,\dots, x_n)\in\F[x_0,x_1,\dots,x_n].$$
		Computations analogous to those in \cite[Lemma 3.2]{NZ-reconstruction} show that $\psi_i(\omega')$ is in $\Stress^\ell_i(\Gamma, p; \mathbb{F})$. Hence, $\psi_i$ is a well-defined map from $\Stress^\ell_i(\Lambda, p'; \mathbb{F}')$  to $\Stress^\ell_i(\Gamma, p; \mathbb{F})$. In fact, these computations show that if $\lambda=\sum_{j=0}^i x_0^j \cdot \lambda_j$ is any element of $\Stress^\ell_i(\Gamma,p;\F)$, where $\lambda_j\in\F[x_1,\dots,x_n]$, then for all $0\leq j\leq i-1$, $\lambda_{j+1}$ is determined from $\lambda_0$ via  equation~\eqref{omega_j}. This holds independently of whether $\lambda$ lies in the image of $\psi_i$.
		
		To see that $\psi_{i-1}(\partial_{c'}\omega')=\partial_c(\psi_i(\omega'))$, note that  $$\partial_c(\psi_i(\omega'))=\partial_{c'}\omega_0+\sum_{j= 1}^i j x_0^{j-1} \cdot\omega_j+\sum_{j=1}^i x_0^j\cdot\partial_{c'}\omega_j.$$
		In particular, $$
		\partial_{c'}\omega_0+\omega_1=\sum_{s=1}^n (1+a_s)\, \partial_{x_s}\omega_0=\sum_{s=1}^n(\partial_{x_s}\omega')\left(\frac{x_1}{1+a_1}, \dots, \frac{x_n}{1+a_n}\right) =(\partial_{c'}\omega')\left(\frac{x_1}{1+a_1},\dots,\frac{x_n}{1+a_n}\right)$$
		is the constant term of $\psi_{i-1}(\partial_{c'} (\omega'))$ expanded as a polynomial in $x_0$. Since both $\psi_{i-1}(\partial_{c'}\omega')$ and $\partial_c(\psi_i(\omega'))$ are elements of $\Stress^\ell_{i-1}(\Gamma,p;\F)$, and since, by the discussion in the previous paragraph, any linear stress on $\Gamma$ is determined by its constant  term, we conclude that these stresses are identical.
		
		It only remains to prove that for $i\geq 1$, $\supp(\psi_i(\omega'))=\skel_{i-1}\big(0*\supp(\omega')\big)$. The definition of $\omega_0$ ensures that every $(i-1)$-face that is in $\supp(\omega')$ is also in $\supp(\psi_i(\omega'))$. Assume now that $G$ is an $(i-2)$-face in $\supp(\omega')$. To complete the proof, we have to show that $G\cup 0$ is in $\supp(\psi_i(\omega'))$.
		Write $\omega'=(\sum_{j=1}^n\ell_jx_j)x_G+\alpha$, where $\ell_j\in \F'$ and no term of $\alpha$ is divisible by $x_G$. (When $i=1$, we have $G=\emptyset$, $\alpha=0$, and $\omega'=\sum_{j=1}^n \ell_jx_j$.)
		The assumption that $G\in \supp(\omega')$ guarantees that at least one of the $\ell_j$ is nonzero.
		We then obtain:
		\begin{eqnarray*}
			\partial_{x_{G\cup 0}}(\psi_i(\omega')) &=& \partial_{x_G} \omega_1(x_1, \dots, x_n)=\partial_{x_G}\left[\sum_{s=1}^n a_s\,\partial_{x_s}\omega_0(x_1, \dots, x_n) \right] \\
			&=& \partial_{x_G}\left[\sum_{s=1}^n a_s\,\partial_{x_s}\left(\omega'\left(\frac{x_1}{1+a_1}, \dots, \frac{x_n}{1+a_n}\right)\right)\right]\\
			&=&	 \partial_{x_G} \left[\sum_{s=1}^n a_s\,\partial_{x_s}  \left( \sum_{j=1}^n \frac{\ell_jx_j}{1+a_j}\cdot \frac{x_G}{\prod_{k\in G} (1+a_k)}\right) \right]\\
			&=&  \frac{1}{\prod_{k\in G} (1+a_k)}\left[ \sum_{j\in G} \frac{2\ell_ja_j}{1+a_j}+\sum_{j\notin G} \frac{\ell_ja_j}{1+a_j}\right].
		\end{eqnarray*}
		Since $\ell_j\in \mathbb{F}'$ and at least one of them is nonzero, and since $a_1,\ldots, a_n$ are algebraically independent over $\F'$, we conclude that $\partial_{x_{G\cup0}} (\psi_i(\omega'))\neq 0$. Thus, $G\cup 0 \in \supp(\psi_i(\omega'))$.
	\end{proof}	
	
	Since the maps $\psi_i$, $1\leq i\leq d-1$, of Lemma~\ref{cone lemma-linear} satisfy $\psi_{i-1}(\partial_{c'}\omega')=\partial_c(\psi_i(\omega'))$, it follows that the restriction of $\psi_i$ to $\Stress^a_i(\Lambda, p';\F')$ defines a map from  $\Stress^a_i(\Lambda, p';\F')$ to $\Stress^a_i(\Gamma, p;\F)$. This implies the following affine version of Lemma~\ref{cone lemma-linear}, which will be used in Sections 4 and 5.

	\begin{lemma} \label{cone lemma-affine}
	Under the assumptions of Lemma~\ref{cone lemma-linear}, for all $1\leq i\leq \lfloor (d-1)/2\rfloor$, there exists a linear map of $\F'$-vector spaces $\psi_i: \Stress^a_i(\Lambda, p';\F') \to \Stress^a_i(\Gamma, p;\F)$ such that for any $\omega'\in\Stress^a_i(\Lambda, p';\F')$, $\supp(\psi_i(\omega'))=\skel_{i-1}\big(0*\supp(\omega')\big)$.
	\end{lemma}

	Consider a $(d-1)$-sphere $\Delta$ with  a generic embedding $p$. In this paper, we often apply Lemmas~\ref{cone lemma-linear} and \ref{cone lemma-affine} to the subcomplexes $\Lambda=\lk(u)$ and $\Gamma=\st(u)=u*\Lambda$ of $\Delta$, where $u\in V(\Delta)$ is a vertex. If $2k<d$, then by the partition of unity theorem for affine stresses \cite{NZ-Aff-Reconstr}, $\Stress^a_k(\Delta, p)=\sum_{u\in V(\Delta)} \Stress^a_k(\st(u), p)$, and hence any affine $k$-stress on $\Delta$ can be expressed as a sum of affine $k$-stresses that live on vertex stars. When $2k=d$, the space $\Stress^a_k(\Delta,p)$ is much more mysterious. The following lemma, whose proof utilizes Lemma~\ref{cone lemma-linear}, sheds some light on this space in the case where $\Delta\in S(k-1,2k-1)$. This lemma will be essential in Section 5.

	\begin{lemma}\label{lm1}
		Let $d=2k$, and let $\Delta\in S(k-1, 2k-1)$. Let $\{a_{s,j} : s\in V(\Delta), j\in [d]\}\subset \R$ be a set of numbers algebraically independent over $\Q$, and let $\F=\Q(a_{s,j} : s\in V(\Delta), j\in [d])$. Define $q:V(\Delta)\to\F^d$ by $q(s)=(a_{s,1},\dots,a_{s,d})$ for all $s\in V(\Delta)$. Then, for every edge $uv\in\Delta$, there exists a stress $\bar{\omega}\in \Stress^a_k(\Delta,q;\F)$ with the following properties: $\supp(\bar{\omega})\subseteq \st(u)\cup\st(v)$, and there exists a face $\sigma\in \supp(\bar{\omega})$ such that $u\in \sigma$ but $\sigma\notin \st(v)$.
	\end{lemma}
	In the proof, we will consider the map $\partial_c$ acting on the stress spaces associated with $\lk(u)$, $\st(u)$, and $\st(v)$. To avoid confusion, we denote these maps by $\phi'_u$, $\phi_u$, and $\phi_v$, respectively.
	
	\begin{proof}
		Consider the $d\times |V(\Delta)|$ matrix $A$, whose columns are labeled by the vertices of $\Delta$ and whose $(j,s)$-entry is $a_{s,j}$ for all $s\in V(\Delta)$, $j\in[d]$. Fix an edge $uv$ and perform the following row operations: for each $j\in [d-1]$, divide the $j$-th row of $A$ by $a_{u,j}$, and divide the $d$-th row of $A$ by $-a_{u,d}$. Then, for each $1\leq j\leq d-1$, add the resulting $d$-th row to the resulting $j$-th row to obtain a new matrix $\tilde{A}$. Thus, for $j\in[d-1]$, the $(j,s)$-entry of $\tilde{A}$ is $\frac{a_{s,j}}{a_{u,j}}-\frac{a_{s,d}}{a_{u,d}}$, while the $(d,s)$-entry of $\tilde{A}$ is $-\frac{a_{s,d}}{a_{u,d}}$. 
		
		These operations do not change the row space of the matrix $A$ over $\F$. As a result, in $\F[x_s: s\in V(\Delta)]$, the ideal generated by the rows of $A$---that is, by $\theta_j=\sum_s a_{s,j}x_s$ for $j\in [d]$---coincides with the ideal generated by the rows of $\tilde{A}$. Consequently, for all $i$, $\Stress^\ell_i(\Delta,q;\F)$ coincides with $\Stress^\ell_i(\Delta,p;\F)$, where the embedding $p$  is given by the columns of $\tilde{A}$. 
		
		Let $p': V(\Delta)\backslash u\to\R^{d-1}$ be the embedding  given by $$p'(s)_j=\left(\frac{a_{s,j}}{a_{u,j}}-\frac{a_{s,d}}{a_{u,d}}\right)/\left(1-\frac{a_{s,d}}{a_{u,d}}\right), \; \text{for}\; j\in[d-1],$$ and define $\F'=\Q\left(p'(s)_j: s\in V(\Delta)\backslash u, j\in[d-1]\right)$. Then $\F=\F'(a_{s,d}, a_{u,j} : s\in V(\Delta), j\in [d-1])$, and the complexes $\Lambda=\lk(u)$ and $\Gamma=u*\lk(u)=\st(u)$, endowed with embeddings $p'$ and $p$, respectively, satisfy the conditions of Lemma~\ref{cone lemma-linear}. Thus, for each $i$, Lemma~\ref{cone lemma-linear} 
		provides  the lifting map $$\psi_{i}: \Stress^\ell_i(\lk(u), p'; \F')\to \Stress^\ell_{i}(\st(u),p; \F).$$

		Consider the map $$\phi'_u: \Stress^\ell_k(\lk(u), p';\F')\to \Stress^\ell_{k-1}(\lk(u),p';\F'), \quad \lambda \mapsto \partial_c \lambda.$$ Since $\lk(u)$ is a $(2k-2)$-sphere, this map is an isomorphism. Moreover, since $\Delta\in S(k-1, 2k-1)$, it follows that $\lk(uv)\in S(k-1, 2k-3)$, and hence, by Remark~\ref{rm: g_k>0}, $$0<g_{k-1}(\lk(uv))=h_{k-1}(\lk(uv))-h_{k-2}(\lk(uv))=h_{k-1}(\lk(uv))-h_{k}(\lk(uv)).$$ Since $v*\lk(uv)=\st(v,\lk(u))$, we conclude that 
	\begin{eqnarray*}\dim \Stress^\ell_{k}(\st(v, \lk(u)), p';\F')&=&h_{k}(\st(v, \lk(u)))=h_k(\lk(uv))\\
		&<& h_{k-1}(\lk(uv))=\dim \Stress^\ell_{k-1}(\st(v, \lk(u)), p';\F').
		\end{eqnarray*} 
		Hence, there exists a linear $(k-1)$-stress $\omega'\in \Stress^\ell_{k-1}(\lk(u), p';\F')$ such that  the support of $\omega'$ is a subcomplex of $v*\lk(uv)=\st(v,\lk(u))$, while the support of $(\phi'_u)^{-1}(\omega')$ contains faces that are not contained in $\st(v, \lk(u))$. We set $\omega=\psi_{k-1}(\omega')\in \Stress^\ell_{k-1}(\st(u), p;\F) \subseteq \Stress^\ell_{k-1} (\Delta, p; \F)$. 
		
		Finally, we consider the maps 
		$$\phi_u: \Stress^\ell_k(\st(u),p; \F)\stackrel{\partial_c}{\to} \Stress^\ell_{k-1}(\st(u),p; \F), \quad \phi_v: \Stress^\ell_k(\st(v),p; \F)\stackrel{\partial_c}{\to} \Stress^\ell_{k-1}(\st(v),p; \F).$$  These maps are also isomorphisms. Observe that, by Lemma~\ref{cone lemma-linear} and our choice of $\omega'$, the support of $\omega$ is a subcomplex of $\st(uv)$, whereas the support of $\phi_u^{-1}(\omega)=\psi_k((\phi'_u)^{-1}(\omega'))$ has some faces that are not contained in $\st(uv)$.  Thus, we may regard $\omega$ as an element of both  $\Stress^\ell_{k-1}(\st(u), p;\F)$ and $ \Stress^\ell_{k-1}(\st(v), p;\F)$. We define $\bar{\omega}:=\iota_u(\phi_u^{-1}(\omega))-\iota_v(\phi_v^{-1}(\omega))\in\Stress^\ell_k(\Delta,p;\F)$, where $\iota_u$ and $\iota_v$ are the natural inclusion maps between the stress spaces on $\st(u)$, $\st(v)$, and on $\Delta$. It is a linear $k$-stress with $\partial_c \bar{\omega}=\omega-\omega=0$, and hence it is an affine $k$-stress.  We show that $\bar{\omega}$ has the desired properties.
		
		If there is a face $\sigma\in \supp(\phi_u^{-1}(\omega))  \cap  \st(u)$ such that $u\in \sigma$ and $\sigma\notin\st(v)$, then $\sigma\notin \supp(\phi_v^{-1}(\omega))$ because $\phi_v^{-1}(\omega)$ lives on $\st(v)$. This implies that $\sigma\in \supp(\bar{\omega})$, as desired. 

		Otherwise, $\phi_u^{-1}(\omega)$ lives on $\st(uv)\cup \lk(u)$. Since it is not supported on $\st(uv)$, there exists a face $\sigma \in \supp(\phi_u^{-1}(\omega))\cap \lk(u)$ with $\sigma\notin \st(uv)$. Thus $\sigma$ is an interior face of the ball $\lk(u)\backslash v$, and hence contains a subset $\tau$ that is a minimal interior face of $\lk(u)\backslash v$. In particular, $\tau$ is a missing face of $\lk(v, \lk(u))=\lk(uv, \Delta)$, and therefore $\partial\overline{\tau} * \overline{uv} \subseteq \Delta$. On the other hand, since $\tau\in \lk(u)$, we have $\tau\cup u\in \Delta$. This implies that either $\tau\cup v$ or $\tau\cup u\cup v$ is a missing face of $\Delta$. In either case, since $\Delta\in S(k-1,2k-1)$, we obtain $\dim \tau \le k-2$. By Lemma~\ref{cone lemma-linear}, the support of $\phi_u^{-1}(\omega)$ is of the form $\skel_{k-1}(C*u)$, where $C$ is the support of $(\phi'_u)^{-1}(\omega')$. Hence $\tau*u\in \supp(\phi_u^{-1}(\omega))$, a contradiction to the assumption that $\phi_u^{-1}(\omega)$ is supported on $\st(uv)\cup \lk(u)$.
	\end{proof}
	
	\begin{remark}
		In the case $k=2$, it is known that if  $\Delta\in S(1,3)$ and $p$ is a generic embedding, then, for every edge $uv\in \Delta$, the framework $(\st(u)\cup \st(v), p)$ is rigid and supports a nontrivial affine $2$-stress. 
		For $k>2$ and $\Delta\in S(k-1, 2k-1)$, an analogous $(k-1)$-skeletal rigidity of the union of the stars of adjacent vertices may fail. Nevertheless, Lemma~\ref{lm1} still allows us to locate a nontrivial affine $k$-stress that lives on this subcomplex.
		\end{remark}
		
	\begin{remark}
		 Lemma~\ref{lm1} holds not only for spheres in $S(k-1, 2k-1)$ but also for certain other $(2k-1)$-spheres under additional conditions on the links of $u$ and $v$.  For example, using the same proof, one can show that, given any $(2k-1)$-sphere $\Delta$ with a generic embedding, there exists a nontrivial affine $k$-stress supported on $\st(u)\cup \st(v)$, where $uv$ is an edge, provided that (1) $g_{k-1}(\lk(uv))>0$, and (2) every face that lies in $\lk(u)\cap \lk(v)$ but not in $\lk(uv)$ has dimension at most $k-2$.
	\end{remark}

	\section{Proof of the first main result}
	We are now in a position to prove the first main result of the paper:
	
	\begin{theorem} \label{main-thm:g=1}
		Let $d\geq 2k$, and let $\Delta\in S(d-k,d-1)$. Assume further that if $d=2k$, then $\Delta\notin S(k-1,2k-1)$. Then $g_k(\Delta)=1$ if and only if one of the following holds: (1) $\Delta=\partial \sigma^{d-k}*\Gamma$, where $\Gamma$ is a $(k-1)$-sphere, or (2) $\Delta=\partial \sigma^j *\partial \sigma^{d-j}$ for some $j$ with $k < j\leq \lfloor d/2\rfloor$.
	\end{theorem}
	
	One direction of the theorem is immediate. Indeed, it is known (and easy to verify) that (1) $h_i(\partial\sigma^j)=1$ for all $0\leq i\leq j$;   (2) if $\Gamma$ is a $(k-1)$-sphere, then $h_k(\Gamma)=1$; and (3) $h_i(\Lambda*\Gamma)=\sum_{j=0}^ih_j(\Lambda)h_{i-j}(\Gamma)$ for any simplicial complexes $\Lambda$ and $\Gamma$ and any $i$. These results imply that if $\Delta$ is one of the complexes described in the theorem, then $g_k(\Delta)=1$. 
	
	In the remainder of this section, we focus on the other direction. Throughout, let $\Delta$ be a sphere satisfying the conditions of Theorem~\ref{main-thm:g=1}. To prove the theorem, we consider the following three cases separately: $d\geq 2k+2$, $d=2k+1$, and $d=2k$.

	\subsection{$d\geq 2k+1$}
	We begin with the case $d\geq 2k+2.$ Recall that, by Proposition~\ref{prop:1and0imply1and0}, if $k\le d/2-1$ and $\Delta$ is a $(d-1)$-sphere with $g_k=1$, then the tail $(g_k, g_{k+1}, \dots, g_{\lfloor d/2\rfloor})$ of the $g$-vector consists of a string of ones followed by a (possibly empty) string of zeros. To complete the proof of Theorem~\ref{main-thm:g=1} in this case, it therefore suffices to verify the following result. 
	\begin{proposition}  \label{prop:d>2k+1}
		Let $k\leq d/2-1$, and let $\Delta \in S(d-k, d-1)$ satisfy $g_k(\Delta)=1$. Suppose there exists an integer $j$ with $k\leq j\leq (d-2)/2$ such that $g_j(\Delta)\neq 0$ but $g_{j+1}(\Delta)=0$. Then either $j=k$, in which case $\Delta=\partial \sigma^{d-k}*\Gamma$ for some $(k-1)$-sphere $\Gamma$, or $j>k$, in which case $\Delta=\partial \sigma^{d-j}*\partial \sigma^{j}$. On the other hand, if no such $j$ exists---that is, if  $g_{\lfloor d/2\rfloor}(\Delta)=1$---then $\Delta=\partial \sigma^{\lfloor d/2\rfloor} * \partial  \sigma^{\lceil d/2\rceil}$.
	\end{proposition}
	\begin{proof}
		Assume first that there exists $k\leq j\leq (d-2)/2$ such that $g_j=1$ but $g_{j+1}=0$.
		Since $g_{j+1}(\Delta)=0$ and $j+1\leq d/2$, $\Delta$ is $j$-stacked. Furthermore, by Corollary~\ref{cor:Nagel}, $m_{d-j}(\Delta)=1$ and $m_{d-\ell}(\Delta)=0$ for all $k\leq \ell \leq j-1$. Our assumptions also imply that $m_{d-\ell}=0$ for all $\ell<k$. Thus, $\Delta$ has a unique missing $(d-j)$-face, which we denote by $\tau$, and no missing faces of dimension greater than $d-j$.  Let $T=\Delta(d-j-1)$ be the $j$-stacked triangulation of $\Delta$ from Theorem~\ref{thm: GLBT}. Then $\tau$ is the unique minimal interior face of this triangulation. Hence, by Lemma \ref{lm:ball-with-one-interior-face}, $T=\overline{\tau}*\lk(\tau, T)$, and
		$$\Delta=\partial T=\partial (\overline{\tau}* \lk(\tau, T))=\partial \overline{\tau}* \lk(\tau, T).$$ 
		Let $\Gamma:=\lk(\tau, T)$; it is a $(j-1)$-sphere, and one easily checks that $g_k(\Delta)=h_k(\Gamma)$. Thus, $h_k(\Gamma)=1$. Observe  that if $j=k$, then any $(k-1)$-sphere $\Gamma$ satisfies $h_k(\Gamma)=1$. However, if $j>k$, then the only $(j-1)$-sphere $\Gamma$ with $h_k(\Gamma)=1$ is the boundary of a $j$-simplex. This completes the proof of the case $g_{j+1}(\Delta)=0$.
		
		We now consider the case $g_{\lfloor d/2\rfloor}(\Delta)=1$. In this case, it suffices to show that $\Delta$ is the join of the boundary complexes of two simplices whose dimensions sum to $d$. The fact that the dimensions of those simplices are $\lfloor d/2\rfloor$ and $\lceil d/2\rceil$  then follows from our assumption that $g_{\lfloor d/2\rfloor}(\Delta) \neq 0$. By McMullen’s formula, $\sum_{v\in V(\Delta)}g_k(\lk (v))=(k+1)g_{k+1}(\Delta)+(d-k+1)g_k(\Delta)=d+2$. Hence there are exactly $d+2$ vertices whose links have $g_k=1$. 
		
		If the total number of vertices is $d+2$, then $\Delta$ is the join of the boundary complexes of two simplices whose dimensions sum to $d$ (see \cite[Section 6.1]{Gru-book}), and the result follows. Hence assume that $\Delta$ has at least $d+3$ vertices, and let $w$ be a vertex with $g_k(\lk(w))=0$. Then $\lk(w)$ is a $(d-2)$-sphere, $g_k(\lk(w))=0$, and $k\leq(d-1)/2$. By Remark \ref{rm: g_k>0} (specifically, its penultimate sentence), applied to $\lk(w)$, it follows that $\lk(w)$ has a missing face of dimension $\geq (d-1)-k+1=d-k$. Thus either $\tau$ or $\tau\cup w$ is a missing face of $\Delta$. However, our assumptions that $\Delta\in S(d-k, d-1)$ and that $g_{\lfloor d/2\rfloor}(\Delta)=1$, together with Corollary~\ref{cor:Nagel}, guarantee that $m_\ell(\Delta)=0$ for all $\ell>\lceil d/2\rceil$. Since $d-k>\lceil d/2\rceil$, it follows that neither $\tau$ nor $\tau\cup w$ can be a missing face, yielding the desired contradiction.
	\end{proof}
	
	We now consider the second case, $d=2k+1$. The argument proceeds in an analogous manner, again relying on McMullen’s formula and the GLBT. We recall that, by the Dehn--Sommerville relations, any $2k$-sphere satisfies $h_k = h_{k+1}$, or equivalently, $g_{k+1}=0$.
	
	\begin{proposition}
		If $\Delta\in S(k+1,2k)$ satisfies $g_k(\Delta)=1$, then $\Delta$ is the join of $\partial \sigma^{k+1}$ and a $(k-1)$-sphere. 
	\end{proposition}
	\begin{proof}
		Let $\Delta$ be a sphere in $S(k+1,2k)$ with $g_k(\Delta)=1$. By McMullen’s formula, $$\sum_{v\in V(\Delta)} g_k(\lk(v))=(k+1)g_{k+1}(\Delta)+(k+2)g_k(\Delta)=k+2.$$ Hence there are exactly $k+2$ vertices whose links satisfy $g_k=1$; for every other vertex $w$, $\lk(w)$ is a $(2k-1)$-sphere with $g_k=0$. It follows from  Remark~\ref{rm: g_k>0} that for such $w$, $\lk(w)$ has a missing face $\tau$ of dimension $\geq 2k-k+1=k+1$. Then either $\tau$ or $\tau\cup w$ would be a missing face of $\Delta$. However, since $\Delta\in S(k+1,2k)$, the latter case is impossible. Therefore, $\tau$ is a missing $(k+1)$-face of $\Delta$. Now, for every vertex $u\in\tau$, the set $\tau\backslash u$ is a missing $k$-face of $\lk(u)$. By Remark~\ref{rm: g_k>0}, this implies $g_k(\lk(u))=1$. Since there are $k+2$ vertices in $\tau$, it follows that the vertices of $\Delta$ whose links satisfy $g_k=1$ are precisely the vertices of $\tau$. Moreover, $\Delta$ has no other missing $(k+1)$-faces; otherwise there would be more than $k+2$ vertices whose links satisfy $g_k=1$. 
		
		We conclude from the above discussion that for every vertex $w$ not in $\tau$, the link of $w$ is a $(k-1)$-stacked $(2k-1)$-sphere and $\tau$ is its unique missing $(k+1)$-face. As in the proof of Proposition~\ref{prop:d>2k+1}, this implies that $\lk(w)=\partial\overline{\tau} * S_w$ for some $(k-2)$-sphere $S_w$. In other words, $\st(w)= \partial\overline{\tau}*w*S_w$. Now, every facet of $\Delta$ contains a vertex that is not in $\tau$. Thus, 
		$$\Delta=\cup_{w\notin \tau} \st(w)=\partial\overline{\tau} *  [\cup_{w\notin \tau} w*S_w].$$
		To summarize, $\Delta$ is the join of two complexes one of which is $\partial\overline{\tau} = \partial \sigma^{k+1}$. Since $\Delta$ is a $2k$-sphere, the remaining factor must be a $(k-1)$-sphere. The result follows.
	\end{proof}
	
	\subsection{$d=2k$}
	As our final case, we consider a sphere $\Delta$ with $g_k(\Delta)=1$ that lies in $S(k, 2k-1)$ but not in $S(k-1,2k-1)$. Let the vertex set be $V=V(\Delta)$, and let $\tau=\{v'_1,\dots,v'_{k+1}\}$ be a missing $k$-face of $\Delta$. Our goal is to show that $\Delta$ is the join of $\partial \overline{\tau}$ with a $(k-1)$-sphere; see Proposition~\ref{prop:d=2u,missing-u-face} below. The proof relies on several lemmas. To start, choose a generic embedding $q : V(\Delta)\to \R^{2k}$ as in Lemma~\ref{lm1}, and let $\omega$ be the unique affine $k$-stress on $(\Delta,q;\F)$, up to scalar multiplication.

	For each $1\leq i\leq k+1$, the face $\tau\backslash v'_i$ is a missing $(k-1)$-face of the $(2k-2)$-sphere $\lk(v'_i)$. Define $\Lambda=\lk(v'_i) \cup \{\tau\backslash v'_i\}$, so that $\Lambda$ is obtained from $\lk(v'_i)$ by adjoining this one missing face. We endow $\Lambda$ (equivalently, $\lk(v'_i)$) with the embedding $p'$ defined in the proof of Lemma~\ref{lm1}, where $v'_i$ plays the role of $u$. The following result continues to use the notation of Section~3; in particular, we define $\F'$ as in the proof of Lemma \ref{lm1}, and we let $\Theta'$ be the set of $2k-1$ linear forms associated with $p'$.  The proof is standard, so we only sketch it.
		
	\begin{lemma} \label{lem:+1}For all $j\leq k$, $\dim (\F'[\Lambda]/(\Theta'))_j=h_j(\Lambda)$.
	\end{lemma}
	\begin{proof} By definition of $\Lambda$, we have $\F'[\Lambda]_j=\F'[\lk(v'_i)]_j$ for all $j\leq k-1$, while for $j=k$, there is a natural surjection  $\F'[\Lambda]_k \twoheadrightarrow \F'[\lk(v'_i)]_k$. Furthermore, since $\lk(v'_i)$ is Cohen--Macaulay, the sequence $\Theta'$ is regular on $\F'[\lk(v'_i)]$. It follows that for all $1\leq r\leq 2k-1$ and $j\leq k$,  multiplication by $\theta'_r$ 
	\[ \cdot \theta'_r: \left(\F'[\Lambda]/(\theta'_1, \dots,\theta'_{r-1})\right)_{j-1} \to \left(\F'[\Lambda]/(\theta'_1, \dots,\theta'_{r-1})\right)_{j}\] 
	is injective, and hence the sequences 
	\[0\to  \left(\F'[\Lambda]/(\theta'_1, \dots,\theta'_{r-1})\right)_{j-1}  \stackrel{\cdot \theta'_r}{\to} \left(\F'[\Lambda]/(\theta'_1, \dots,\theta'_{r-1})\right)_{j} \to \left(\F'[\Lambda]/(\theta'_1, \dots,\theta'_{r-1}, \theta'_r)\right)_{j} \to 0\]
	are exact. Standard Hilbert-function calculations then give the claim.
	\end{proof}
	
	By Lemma \ref{lem:+1}, $$\dim (\F'[\Lambda]/(\Theta'))_k=h_k(\Lambda)=h_k(\lk(v'_i))+1=h_{k-1}(\lk(v'_i))+1,$$ while $\dim (\F'[\Lambda]/(\Theta'))_{k-1}=h_{k-1}(\lk (v'_i))$. Thus the affine stress space $\Stress^a_k(\Lambda,p';\F')$ is at least $1$-dimensional. Lemma~\ref{cone lemma-affine} then implies that $\Stress^a_k(\Lambda*v'_i,p;\F)$ is also at least $1$-dimensional. On the other hand, since the subcomplex $\st(v'_i)\cup \{\tau\backslash v'_i\}$ of $\Delta$ and $\Lambda*v'_i$ have the same $(k-1)$-skeleton, and since $\omega$ is the unique affine $k$-stress on $(\Delta,q;\F)$ up to scalar multiplication, this stress space is in fact $1$-dimensional and is spanned by $\omega$. We obtain the following result.

	\begin{lemma}\label{lm: support of the unique stress}
		For every $1\leq i\leq k+1$, the stress $\omega$ lives on $\st(v'_i)\cup \{\tau\backslash v'_i\}$. Furthermore, for each $1\leq i\leq k+1$, the support of $\omega$ is of the form $\supp(\omega)=\skel_{k-1}(v'_i*C_i)$ for some subcomplex $C_i$ of $\lk(v'_i)\cup \{\tau\backslash v'_i\}$.
	\end{lemma}

	Lemma~\ref{lm: support of the unique stress} leads to the following structural result about $\Delta$.
	\begin{lemma}\label{lm: partial support}
		If a $(k-1)$-face $G\in \Delta[V\backslash \tau]$ participates in $\omega$, then  $\partial\overline{\tau}*\overline{G}$ is a subcomplex of $\Delta$.
	\end{lemma}
	\begin{proof} Pick any $v'_i\in\tau$. By Lemma~\ref{lm: support of the unique stress}, the stress $\omega$ lives on $\st(v'_i)\cup \{\tau\backslash v'_i\}$. Hence if $G\in \supp(\omega)$ and $\tau\cap G=\emptyset$, then $G\in \lk(v'_i)$. Moreover, the second part of Lemma~\ref{lm: support of the unique stress} implies that for every $v\in G$, the $(k-1)$-face $(G\backslash v)\cup v'_i$ also lies in the support of $\omega$. Now pick any $v'_j\in\tau$ with $j\neq i$. Since $\omega$ also lives on $\st(v'_j)\cup\{\tau\backslash v'_j\}$, it follows that $(G\backslash v)\cup v'_i$ must be in the star of $v'_j$. This implies that replacing any two vertices of $G$ with any two vertices of $\tau$ yields a $(k-1)$-face of $\Delta$. By Lemma~\ref{lm: support of the unique stress}, this face also participates in $\omega$. Iterating this argument, we conclude that $\skel_{k-1}(\partial\overline{\tau}*\overline{G})$ is a subcomplex of $\Delta$. 
		
		To complete the proof, observe that any two missing $k$-faces $K, K'$ of $\Delta$ must be disjoint. (Indeed, if $v\in K\cap K'$, then $\Stress^a_k\big(\st(v)\cup \{K\backslash v, K'\backslash v\}\big)$ has dimension at least $2$, contradicting the assumption that $g_k(\Delta)=1$.) Since $|\tau\cup G|=2k+1<2k+2$, it follows that $\tau$ is the unique missing $k$-face of $\Delta[\tau\cup G]$. Moreover, because $\skel_{k-1}(\partial\overline{\tau}*\overline{G})$ is a subcomplex of $\Delta$, we conclude that $\skel_{k}(\partial\overline{\tau}*\overline{G})$ is also a subcomplex of $\Delta$. Finally, since $\Delta$ has no missing faces of dimension greater than $k$, it follows that $\partial\overline{\tau}*\overline{G}$ is a subcomplex of $\Delta$.
	\end{proof}
	
	We are ready to complete the proof of Theorem~\ref{main-thm:g=1} by verifying the following result.
	\begin{proposition} \label{prop:d=2u,missing-u-face}
		Let $\Delta\in S(k, 2k-1)\backslash S(k-1,2k-1)$ be  a sphere with $g_k=1$, and let $\tau$ be a missing $k$-face of $\Delta$.
		Then $\Delta$ is the join of $\partial\overline{\tau}$ and a  $(k-1)$-sphere. 
	\end{proposition}
	\begin{proof}
		We continue with the notation introduced above; in particular, we write $\tau=\{v'_1,\ldots,v'_{k+1}\}$. We also let $V\backslash \tau=\{v_1,\dots, v_m\}$. Then for each $i=1,\ldots,k+1$, the link of $\tau\backslash v'_i$ is a $(k-1)$-sphere, which we denote by $S_i$. Our goal is to show that $S_1=\dots=S_{k+1}$.
		
		By Lemma~\ref{lm: partial support}, if a face $G\in \Delta[V\backslash \tau]$ participates in $\omega$, then $G$ must belong to the intersection of $S_1,\dots,S_{k+1}$. Consequently, if not all $(k-1)$-spheres $S_1,\dots,S_{k+1}$ are identical, then $(\supp(\omega))[V\backslash \tau]$ is a proper (possibly empty) subcomplex of $S_1$. It follows that either no $(k-1)$-face of $\Delta[V\backslash \tau]$ participates in $\omega$ or there exists a $(k-2)$-face $H\in \Delta[V\backslash \tau]$ that is contained in a unique $(k-1)$-face $G \in (\supp(\omega))[V\backslash \tau]$. 
		
		We claim that neither of these two cases can occur. We begin with the latter case. By relabeling the vertices if necessary, we may assume that $H=\{v_1,\dots,v_{k-1}\}$ and $G=\{v_1,\dots,v_k\}$. Then the stress $\omega$ can be written in the form
		$$\omega = x_H \left[\sum_{i=1}^{k-1} \alpha_i x_{v_i} + \beta x_{v_k}+ \sum_{j=1}^{k+1}\gamma_j x_{v'_j}\right] + \sum_{x_H \nmid \mu} \delta_\mu \mu.$$
		Consequently, $\partial_{x_H}\omega=\sum_{i=1}^{k-1} 2\alpha_i x_{v_i} + \beta x_{v_k}+ \sum_{j=1}^{k+1}\gamma_j x_{v'_j}$ is a nontrivial affine $1$-stress on $\Delta$. However, the corresponding affine dependence involves only $2k+1$ points in $\R^{2k}$. Since the embedding is generic, these points cannot be affinely dependent, and therefore no such affine $1$-stress can exist.

		The former case can be handled in a similar manner. 
		For each $(k-2)$-face $H'$ in $\supp(\omega)$, consider the size of $H'\cap\{v_1,\dots,v_m\}$, and let $H$ be a face for which this size is maximized. 
		Say, $H=\{v_1,\dots,v_\ell,v_1',\dots,v_{k-1-\ell}'\}$ for some $0\leq \ell\leq k-1$. Then $\omega$ can be written as
		$$\omega = x_{H'}\left[\sum_{i=1}^{\ell} \alpha_i x_{v_i} + \sum_{j=1}^{k+1} \gamma_jx_{v'_j}\right]+ \sum_{x_{H'}\nmid \mu}\delta_\mu \mu.$$ Taking the partial derivative $\partial_{x_{H'}}$, we  obtain an affine dependence among $\ell+(k+1)\leq 2k$ generic points in $\R^{2k}$. This is again impossible.
		
		Hence all the spheres $S_1, \dots, S_{k+1}$ are the same sphere $S$, and $\partial\overline{\tau} * S$ is a subcomplex of $\Delta$. Since $\partial\overline{\tau} * S$ is a $(2k-1)$-sphere, we must have $\Delta=\partial\overline{\tau} * S$.
	\end{proof}

	\section{Proof of the second main result}
	
	The goal of this section is to characterize all $5$-spheres in $S(2,5)$ with $g_3=1$; see Theorem~\ref{main-thm: S(2,5)}. Throughout this section, let $\Delta$ be a sphere in $S(2, 5)$ with $g_3(\Delta)=1$, equipped with a generic embedding $q$ as in Lemma~\ref{lm1}. Let $\omega\in \Stress^a_3(\Delta,q;\F)$ be the unique affine $3$-stress on $\Delta$, up to scalar multiplication. We will repeatedly use the fact that, by Lemma~\ref{lm1}, for every edge $ab\in\Delta$, the support of $\omega$ is contained in $\st(a)\cup\st(b)$. Below, we exploit this observation to establish several properties of the support of $\omega$, which, in turn, enable us to prove results about the structure of $\Delta$.
	
	\begin{lemma}\label{lm: vertices in the support}
		Every vertex of $\Delta$ is in $\supp(\omega)$. 
	\end{lemma}
	\begin{proof}
		By Lemma~\ref{lm1}, for every edge $ab\in \Delta$, the support of $\omega$ is contained in $\st(a)\cup \st(b)$; furthermore, there exist faces $F_a,F_b\in \supp(\omega)$ such that $a\in F_a$, $b\in F_b$, but $F_a, F_b\notin \st(ab)$. In particular, $a$ and $b$ are in the support of $\omega$. As every vertex is contained in some edge of $\Delta$, this implies that every vertex of $\Delta$ participates in $\omega$.
	\end{proof}
	
	\begin{lemma}\label{lm: edges in the support}
		Let $ab\in \Delta$, and let $y\in V(\Delta)\backslash\{a,b\}$ be a vertex not in $\lk(ab)$. Then at least one of the pairs $ay$ and $by$ is an edge contained in $\supp(\omega)$. In particular, every missing $2$-face of $\Delta$ has at least two edges contained in $\supp(\omega)$.
	\end{lemma}
	\begin{proof}
		View $\omega$ as an affine $3$-stress on $\st(a)\cup \st(b)$. Following the notation of the proof of Lemma~\ref{lm1}, we can write  $\omega = \iota_a(\phi_a^{-1}(\alpha))-\iota_b(\phi_b^{-1}(\alpha))$ for some linear $2$-stress $\alpha$ on $\st(ab)$. By Lemma~\ref{lm: vertices in the support}, we have $y\in\supp(\omega)$, and hence $y$ must lie in the support of either $\phi_a^{-1}(\alpha)$ or $\phi_b^{-1}(\alpha)$. Assume, without loss of generality, that $y$ lies in the support of $\phi_a^{-1}(\alpha)=\psi_a((\phi_a')^{-1}(\alpha'))$. Here, following the notation of the proof of Lemma~\ref{lm1}, $\psi_a$ is the lifting map from $(\lk(a), p')$ to $(\st(a), p)$, the map $\phi_a': \Stress^\ell_3(\lk(a), p')\to \Stress^\ell_2(\lk(a), p')$ is an isomorphism, and $\alpha'$ is a linear $2$-stress on $(\lk(a), p')$. Then, by {\color{blue}Lemma~\ref{cone lemma-linear}}, $ay$ is an edge and also lies in the support of $\phi_a^{-1}(\alpha)$. Since $y\notin \lk(ab)$, we have $aby\notin\Delta$, and hence $ay\notin \st(b)$. Thus $ay$ does not lie in the support of $\phi_b^{-1}(\alpha)$. Consequently, $ay\in \supp(\omega)$, as desired. 
		
		Now assume that $aby$ is a missing $2$-face. Applying the first statement to each edge of $aby$ and the opposite vertex shows that, for every vertex of $aby$, at least one of the two edges of $aby$ incident to that vertex lies in $\supp(\omega)$. This implies that at least two edges of $aby$ are contained in $\supp(\omega)$, proving the second statement.
	\end{proof}
	
	We now use the properties established above to show that the graph of $\Delta$ must be a complete graph. We begin by proving the following weaker result.
	\begin{lemma} \label{lm:multipartite}
		The graph of $\Delta$ is a complete multipartite graph with at least six parts.
	\end{lemma}
	\begin{proof}
	Assume that $ab$ is not an edge. We claim that the set of neighbors of $a$ coincides with that of $b$. Indeed, suppose that $y$ is a neighbor of $a$ but not of $b$. Since $ay\in \Delta$ and $b\notin \lk(ay)$, Lemma~\ref{lm: edges in the support} implies that at least one of $ab$ and $by$ is an edge contained in $\supp(\omega)$. This is impossible, since neither $ab$ nor $by$ is an edge.
	
	Let $I$ be a maximal independent set of $\Delta$ with respect to inclusion. The above claim implies that each vertex in $I$ is connected to every vertex of $\Delta$ not in $I$. Since this holds for all maximal independent sets, the maximal independent sets partition the vertex set of $\Delta$, and the graph of $\Delta$ is the complete multipartite graph with these parts. Moreover, since $\dim\Delta=5$, there are at least six parts.
	\end{proof}
	
	To show that $\Delta$ has a complete graph, we need a few additional definitions.
	An edge $uv$ of $\Delta$ is called {\em contractible} if $uv$ is not contained in any missing face. If $uv$ is contractible, then one can contract $uv$ and obtain a new simplicial complex $\Delta'$ by replacing the vertices $u$ and $v$ with a new vertex $u'$.  Every face $\tau$ containing $u$, $v$, or both is then replaced by $(\tau\backslash\{u,v\})\cup \{u'\}$, while all other faces remain unchanged. The resulting complex $\Delta'$ is a (homology) $5$-sphere; see \cite[Proposition 2.3]{NevoNovinsky}. 
	
	In what follows, we write $(u_1, u_2,\dots, u_j)$ to denote the {\em $j$-cycle} with edges $u_1u_2, \dots, u_{j-1}u_j, u_ju_1$. We also write $\widetilde{H}_i(\Delta)$ to denote the $i$-th reduced homology group of $\Delta$ with coefficients in $\Z/2\Z$.

	\begin{lemma} \label{lem:contr-edge}
		If $\Delta$ has a contractible edge, then $\Delta$ is the join of three $3$-cycles.
	\end{lemma}
	\begin{proof}
		Let $uv$ be a contractible edge of $\Delta$, and let $\Delta'$ be the $5$-sphere obtained from $\Delta$ by contracting $uv$ to a new vertex $u'$. By \cite[page 172]{Nevo07}, $h_i(\Delta)=h_i(\Delta')+h_{i-1}(\lk(uv))$ for all $i\geq 1$. This implies that $g_3(\Delta')=g_3(\Delta)-g_2(\lk(uv, \Delta))$. Since $\Delta\in S(2,5)$, we have $\lk(uv)\in S(2,3)$. By Remark~\ref{rm: g_k>0}, $g_2(\lk(uv))=1$, and hence $g_3(\Delta')=0$. By Theorem~\ref{thm: GLBT}, $\Delta'$ is therefore $2$-stacked. 
		
		Moreover, by \cite[Theorem 1.3]{NevoNovinsky}, the conditions $g_2(\lk(uv))=1$ and $\lk(uv)\in S(2,3)$ imply that $\lk(uv)=\partial \overline{\tau} *C$, where $\overline{\tau}$ is a $2$-simplex and $C$ is a cycle. Note that $\tau\notin \Delta'$; otherwise, $\tau$ would be contained in a missing face of dimension greater than $2$ in $\Delta$. Hence $\partial \overline{\tau}$ is an induced $3$-cycle of $\Delta'$. Therefore, $\Delta'$ is $2$-stacked but not stacked, and so it must have at least one missing $4$-face. 
		
		Let $\{a,b,y,z,u'\}$ be a missing $4$-face of $\Delta'$, and write $\sigma=\{a,b,y,z\}$. Then $\Delta[\sigma\cup uv]$ contracts to $\Delta'[\sigma\cup u']=\partial \overline{\sigma\cup u'}$; in particular, $\sigma\in \Delta$. We claim that $K:=\Delta[\sigma\cup uv]$ is the join of two $3$-cycles. First, observe that the subcomplex of $\Delta$ induced by $V(\Delta)\backslash(\sigma\cup uv)$ and the subcomplex of $\Delta'$ induced by $V(\Delta')\backslash(\sigma\cup u')$ are identical. Hence, by Alexander duality \cite[Theorem 71.1]{Munkres}, $\widetilde{H}_3(K) \cong \widetilde{H}_3(\Delta'[\sigma\cup u'])\neq 0$. 
		
		Now $K$ has only six vertices, and all missing faces of $K$ have dimension at most $2$. Let $K^*$ be the combinatorial Alexander dual of $K$. By another form of Alexander duality \cite{BjornerTancer}, $\widetilde{H}_3(K)\neq 0$ and $f_0(K)=6$ imply that $\widetilde{H}_{6-3-3}(K^*)\neq 0$, and so $K^*$ is disconnected. On the other hand, the facets of $K^*$ are precisely the complements of the missing faces of $K$. Since all missing faces of $K$ have dimension at most $2$, every facet of $K^*$ has dimension at least $2$. Hence every connected component of $K^*$ contains a $2$-simplex and therefore at least three vertices. Because $K^*$ is disconnected and has only six vertices, it follows that $K^*$ is the disjoint union of two triangles. Consequently, $K$ is the join of two $3$-cycles---the boundaries of these two triangles. 
		
		Since $uv$ is not contained in any missing $2$-face, we may assume, without loss of generality, that $\Delta[\sigma\cup uv]=(a,b,u)*(y,z,v)$. Thus $\lk(uv, \Delta)$ contains the $4$-cycle $(a,y,b,z)$ as an induced subcomplex. Since $\lk(uv, \Delta)=\partial \overline{\tau}*C$ for some cycle $C$, it follows that $\lk(uv, \Delta)=\partial \overline{\tau}*(a,y,b,z)$. Furthermore, since every missing face of dimension $\geq 4$ in $\Delta'$ contains $u'$, this argument shows that any other missing $4$-face $\rho\cup u'$ would produce a distinct induced $4$-cycle in $\lk(uv,\Delta)$ on the vertex set $\rho$, contradicting the fact that $\lk(uv,\Delta)=\partial\overline{\tau}*(a,y,b,z)$ has a unique induced $4$-cycle, namely $(a,y,b,z)$. Therefore, $\sigma\cup u'$ is the only missing $4$-face of $\Delta'$.
		
		Similarly, if $\Delta'$ has a missing $5$-face, then it must be of the form $\sigma'\cup u'$. Consider the complex $\Delta[\sigma'\cup uv]$, which contracts to $\Delta'[\sigma'\cup u']=\partial \overline{\sigma'\cup u'}$. The same argument as above shows that 		
		$\Delta[\sigma'\cup uv]$ is a simplicial complex on seven vertices with $\widetilde{H}_4\neq 0$ and no missing faces of dimension greater than $2$. An application of Alexander duality similar to the one in the $6$-vertex case shows that no such $7$-vertex complex can exist. Hence $\Delta'$ has no missing $5$-faces.
	
		Therefore, since $\sigma\cup u'$ is the unique missing $4$-face of $\Delta'$ and $\Delta'$ has no missing $5$-faces, it is the unique minimal interior face of the $2$-stacked $6$-ball $\Delta'(2)$. By Lemma \ref{lm:ball-with-one-interior-face}, the $2$-stacked $6$-ball $\Delta'(2)$ is the join of the $4$-simplex $\overline{\sigma\cup u'}$ and its link, which is a $1$-sphere and hence a cycle. Since $\partial \overline{\tau}\subseteq \lk(uv)\subseteq \Delta$, this cycle must be $\partial\overline{\tau}$. It follows that $\Delta=\partial\overline{\tau}*\Delta[\sigma\cup uv]=\partial\overline{\tau}* (a,b,u)*(y,z,v)$, as desired.
	\end{proof}
	
	\begin{proposition}\label{prop: complete graph}
		The graph of $\Delta$ is a complete graph. Moreover, every edge in the support of $\omega$ is contained in at most one missing $2$-face.
	\end{proposition}
	\begin{proof}
	Recall from Lemma~\ref{lm:multipartite} that the graph of $\Delta$, $G(\Delta)$,  is a complete multipartite graph. Thus, to show that $G(\Delta)$ is a complete graph, it suffices to show that each part of $G(\Delta)$ consists of a single vertex.
	
	First, we prove that at most one part of $G(\Delta)$ can have size greater than $1$. Assume, for contradiction, that there are at least two such parts. Then there exist four vertices $a,b,y,z$ such that the subcomplex induced by these vertices is the $4$-cycle $C=(a,b,y,z)$.  Since $yz$ is not in $\st(a)\cup\st(b)$, it follows that $yz\notin \supp(\omega)$. By a similar argument applied to other edges of $C$, we conclude that none of the edges of $C$ are in $\supp(\omega)$. On the other hand, since $z\in \supp(\omega)$ and $z\notin \st(b)$, 
	Lemma~\ref{lm: edges in the support} implies that $az$ lies in $\supp(\omega)$,  a contradiction. 
		
	It remains to consider the case where exactly one part of $G(\Delta)$, say $V_1$, has size larger than $1$. Let $a$ and $y$ be vertices of $V_1$, and let $b$ and $z$ be vertices from other parts of $G(\Delta)$. The subgraph of $\Delta$ induced by these four vertices is the union of two $3$-cycles, $(a,b,z)$ and $(b,y,z)$. Consider $\st(a)\cup \st(b)\supseteq \supp(\omega)$. Since $y\in \st(b)$, $y\notin \st(a)$, and $y\in \supp(\omega)$, Lemma~\ref{lm: edges in the support} shows that $by\in\supp(\omega)$. By symmetry, all edges $ab$, $by$, $yz$, and $za$ lie in $\supp(\omega)$. Now, since $\supp(\omega)$ is also a subcomplex of $\st(b)\cup \st(y)$ and since $az\notin \st(y)$, it follows that $az$ is in $\st(b)$, so $abz\in\Delta$. In other words, the edge $ab$, together with any vertex in $V(\Delta)\backslash (V_1\cup b)$, forms a $2$-face of $\Delta$. Hence, $ab$ is not contained in any missing $2$-face and is therefore a contractible edge.  Lemma~\ref{lem:contr-edge} then implies that $\Delta$ is the join of three $3$-cycles, and consequently, $G(\Delta)$ is a complete graph.	This concludes the proof of the first part of the proposition.
		
		Finally, suppose an edge $uv$ is contained in two missing faces $uvy$ and $uvz$. Since $yz$ is an edge, we have $\supp(\omega)\subseteq\st(y)\cup \st(z)$. However, $uv$ is contained in neither $\st(y)$ nor $\st(z)$, and hence $uv\notin \supp(\omega)$. This completes the proof.
	\end{proof}
	
	We are now ready to prove the second main result of the paper:
	\begin{theorem}\label{main-thm: S(2,5)}
		Let $\Delta\in S(2,5)$ be a sphere with $g_3(\Delta)=1$. Then $\Delta$ is the join of three $3$-cycles.
	\end{theorem}

	\begin{proof}
		Let $n=f_0(\Delta)$. By Proposition~\ref{prop: complete graph}, $f_1(\Delta)=\binom{n}{2}$. The assumption that $g_3(\Delta)=1$ then implies that $f_2(\Delta)=5{n \choose 2}-15n+36$.
		
		Assume that $\Delta$ is not the join of three $3$-cycles. By Lemma~\ref{lem:contr-edge}, every edge of $\Delta$ is contained in at least one missing $2$-face. Hence $$3\left(5{n \choose 2}-15n+36\right)=3f_2=\sum_{e\in \Delta} f_0(\lk(e))\leq \binom{n}{2}(n-3).$$
		Simplifying this inequality yields  $n^3-19n^2+108n-216 \geq 0$, and therefore $n\geq 11$. 
		
		On the other hand, by Lemma~\ref{lm: edges in the support}, every missing $2$-face has at least two edges in $\supp(\omega)$. By Proposition~\ref{prop: complete graph}, each of these two edges is not contained in any other missing $2$-face. Thus, $2m_2(\Delta)\leq f_1(\Delta)={n \choose 2}$. Since $m_2={n \choose 3}-f_2={n \choose 3}-5{n\choose 2}+15n-36$, we obtain $${n \choose 3}-5{n\choose 2}+15n-36 \leq \frac{1}{2}{n \choose 2}.$$ Simplifying this inequality yields $h(n):=2n^3-39n^2+217n-432\leq 0$. For $n\geq 9$, the function $h(n)$ is strictly increasing, and since $h(12)=12>0$, it follows that $n\leq 11$. We conclude that $\Delta$ has exactly $11$ vertices. 
		
		Since $\Delta$ has a complete graph, a direct computation yields $f_1(\Delta)=55$, $f_2(\Delta)=146$, and $m_2(\Delta)=\binom{n}{3}-f_2=19$. Also, by Lemma~\ref{lem:contr-edge}, each edge of $\Delta$ is contained in at least one missing $2$-face. Since $19\cdot 3=55+2$, we obtain that, with the exception of at most two edges, each edge of $\Delta$ is contained in a unique missing $2$-face.
		It follows that there exists a vertex $v$ such that each edge incident to $v$ lies in a unique missing $2$-face. Then $\lk(v)$ is a $4$-sphere with $10$ vertices whose graph contains exactly five missing edges $e_1, \dots, e_5$ and they are pairwise disjoint. Hence $\lk(v)$ is a subcomplex of the octahedral $4$-sphere.  Since $\lk(v)$ is itself a $4$-sphere, it must be precisely the octahedral $4$-sphere. In particular, every triple of vertices of $\lk(v)$ that does not contain one of the edges  $e_i$ is a $2$-face. We conclude that every missing $2$-face of $\Delta$ contains one of the edges $e_i$. But then the total number of missing $2$-faces is at most $5+2=7<19$, which yields a contradiction.
	\end{proof}

   \section{Open problems} 
	We close the paper with a few open problems.
 First, it is natural to ask whether one can remove the assumptions on the dimensions of missing faces and provide a characterization of {\em all} $(d-1)$-spheres with $g_k=1$. To this end, and in view of the GLBT, we propose the following conjecture.

   \begin{conjecture}
   	Let $k\geq 2$, $d\geq 2k$, and let $\Delta$ be a $(d-1)$-sphere with $g_k(\Delta)=1$. Then there exists a $d$-dimensional cell complex $\mathcal{C}$ with the following properties:\footnote{The complex $\mathcal{C}$ is more general than a CW complex because one of the cells of $\C$ may not be homeomorphic to a ball.} 
		\begin{enumerate}
		\item all faces of $\mathcal{C}$ of dimension $\leq d-k$ are faces of $\Delta$; furthermore, all faces of $\mathcal{C}$ except one $d$-face are simplices;
		\item the one exceptional $d$-face is a homology $d$-ball whose boundary is a sphere in $S(d-k,d-1)$ with $g_k=1$;
		\item every two faces  of $\mathcal{C}$ intersect in a common (possibly empty) face;
		\item the geometric realization  of $\mathcal{C}$ is a homology $d$-ball, and the boundary complex of $\mathcal{C}$ is $\Delta$.
		\end{enumerate}
	Furthermore, if $\Delta$ is the boundary complex of a simplicial $d$-polytope, then the exceptional $d$-cell is also the boundary complex of a simplicial $d$-polytope.
		\end{conjecture}

	It is not hard to see that if such a cell complex exists, then $g_k(\Delta)=1$. The converse direction appears to be much harder: the case $k=2$ follows from \cite{NevoNovinsky}, but all other cases remain open.
   
	While in Theorem~\ref{main-thm: S(2,5)}, we characterized all spheres in $S(2,5)$ with $g_3=1$, for $k>3$, the question of which spheres in $S(k-1, 2k-1)$ can have $g_k=1$ remains open. 
	Recall the definition of $K(i, d-1)\in S(i, d-1)$: for fixed integers $d$ and $i$,  write $d=qi+r$ with $1\leq r\leq i$, and define $K(i, d-1):= (\partial \sigma^i)^{*q}*\partial \sigma^r$. In particular, the unique sphere in $S(2, 5)$ with $g_3=1$ is $K(2,5)$.
   
	Using the fact that $h_j(\partial{\sigma^i})=1$ for all $0\leq j\leq i$,	together with the identity $h_j(\Delta*\Gamma)=\sum_{0\leq \ell \leq j} h_\ell(\Delta)h_{j-\ell}(\Gamma)$ for all $j$, one can easily prove the following result. 
   \begin{lemma}\label{lm: unsymmetric g-vector}
   	Let $2i < d\leq 3i$. Then $$g_j(K(i, d-1))=\begin{cases}
   		j+1 & \mbox{ if $0\leq j\leq r$}\\
   		r+1 & \mbox{ if $r+1 \leq j\leq i$}\\
   		d+1-2j & \mbox{ if $i+1\leq j \leq d/2$}
   	\end{cases}.$$
   	In particular, for $k\geq 3$, we have $g_k(K(i, 2k-1))=1$ whenever $\frac{2k}{3}\leq i< k$.
   \end{lemma}

	In light of Theorem~\ref{main-thm: S(2,5)}, it is natural to pose the following question.
	
	\begin{question}Let $k\geq 4$, and let $\Delta\in S(k-1, 2k-1)$ be a sphere with $g_k(\Delta)=1$. Must $\Delta$ be a join of boundary complexes of simplices?
	\end{question}
	
	Finally, we note that among all spheres in $S(i, d-1)$, the sphere $K(i, d-1)$ simultaneously minimizes all the $f$- and $h$-numbers; see \cite{GoffKleeN, Nevo2009}.
	This observation leads to another natural question.
	\begin{question} In $S(i, d-1)$, does $K(i, d-1)$ simultaneously minimize all the $g$-numbers? If so, is it the unique minimizer?
	\end{question}
\noindent In the class of PL flag $(d-1)$-spheres, $K(1, d-1)$---namely the octahedral $(d-1)$-sphere---is indeed the unique minimizer of all $g$-numbers \cite{NZ-Aff-Reconstr}. All other cases remain open.
   
\section*{Acknowledgments}
We are grateful to the referees for their careful reading of the paper and for many suggestions that significantly improved the exposition.

{\small
	\bibliography{refs}
	\bibliographystyle{plain}
}
\end{document}